\documentclass[11pt]{amsart}
\usepackage{amsmath,amscd,amssymb,amsfonts,amsthm,hyperref}

\usepackage[cmtip,matrix,arrow]{xy}
\usepackage{tikz-cd}
\usepackage{tikz}
\usetikzlibrary{decorations.markings}
\usetikzlibrary{arrows}

\newtheorem{theorem}{Theorem}[section]

\newtheorem{proposition}[theorem]{Proposition}

\newtheorem{lemma}[theorem]{Lemma}
\theoremstyle{definition}    
\newtheorem{definition}[theorem]{Definition}
\theoremstyle{remark}

\newtheorem{remark}[theorem]{Remark}
\newtheorem{remarks}[theorem]{Remarks}
\newtheorem{example}[theorem]{Example}
\newtheorem{examples}[theorem]{Examples}
\renewcommand{\L}{\mathcal{L}}
\renewcommand{\O}{\mathcal{O}}

\newcommand{\ca}{\mathcal}

\newcommand{\F}{\mathcal{F}}

\newcommand{\R}{\mathbb{R}}

\newcommand{\SU}{\on{SU}}

\newcommand{\Z}{\mathbb{Z}}

\newcommand\pt{\on{pt}}

\newcommand\lie[1]{\mathfrak{#1}}
\renewcommand{\k}{\lie{k}}
\newcommand{\h}{\lie{h}}

\renewcommand{\a}{\mathsf{a}}

\newcommand{\on}{\operatorname}

\newcommand{\Ad}{ \on{Ad} }
\newcommand{\ad}{\on{ad}}
\newcommand{\Hol}{ \on{Hol} }

\renewcommand{\ker}{ \on{ker}}

\newcommand{\SO}{ \on{SO}}

\newcommand{\sz}{\mathsf{s}}
\newcommand{\tz}{\mathsf{t}}

\newcommand\qu{/\kern-.7ex/} 

\newcommand{\lra}{\longrightarrow}

\newcommand{\ra}{\rightarrow}

\renewcommand{\d}{{\mbox{d}}}

\newcommand{\f}{\frac}

\newcommand{\p}{\partial}

\newcommand\hh{{\f{1}{2}}}

\newcommand{\ti}{\tilde}
\newcommand{\eeq}{\end{eqnarray*}}
\newcommand{\beq}{\begin{eqnarray*}}

\newcommand{\pr}{\on{pr}}

\newcommand{\mf}{\mathfrak}
\newcommand{\rra}{\rightrightarrows}
\renewcommand{\subset}{\subseteq}
\renewcommand{\supset}{\supseteq}
\newcommand{\wt}{\widetilde}
\newcommand{\Ra}{\Rightarrow}

\newcommand{\GL}{\on{GL}}

\newcommand{\sI}{\mathsf{I}}
\newcommand{\sU}{\mathsf{U}}
\begin{document}

\title{On the integration of transitive Lie algebroids}

\author{Eckhard Meinrenken}
\address{Mathematics Department\\ University of Toronto\\40 St George Street\\Toronto, ON M5S2E4}
\email{mein@math.toronto.edu}
\begin{abstract}
We revisit the problem of integrating Lie algebroids $A\Rightarrow M$ to Lie groupoids $G\rra M$, for the special case 
that the Lie algebroid $A$ is \emph{transitive}.  We obtain a geometric explanation of the Crainic-Fernandes obstructions for this situation, and an explicit construction of the integration whenever these obstructions vanish. We also indicate an extension of this approach to regular Lie algebroids. 
\end{abstract}
\maketitle

\section{Introduction}
Lie's Third Theorem states that every finite-dimensional Lie algebra  admits an integration to a unique connected, simply connected Lie group. Lie himself had proved the integrability to \emph{local} Lie groups; the full version was obtained by \'{E}lie Cartan \cite{car:the,car:tro}. Classical textbook proofs 
obtain Lie's Third Theorem as a consequence of rather difficult structure theorems in Lie theory; a more elementary argument was given by Gorbatsevich \cite{gor:con}. 

Following the introduction of Lie algebroids by Pradines \cite{pra:th,pra:tro} in the late 1960s, the generalization of Lie's Third Theorem to the integration of Lie algebroids $A\Rightarrow M$ to Lie groupoids $G\rra M$ was an open question for many years. (See Section \ref{sec:basics} below for some background material on Lie algebroids and Lie groupoids.)
In special cases, such as tangent bundles to foliations, actions of Lie algebras on manifolds, or  smooth families of Lie algebras, the integrability was established through explicit constructions.  
A counter-example, disproving integrability in the general case, was given by Almeida and Molino \cite{alm:sui} in 1985. These authors observed that the Lie algebroid $TM\times \R\Ra M$, with bracket on sections
\[ [X+f,Y+g]=[X,Y]+L_X g-L_Y f+\sigma(X,Y)\]
defined by a closed 2-form $\sigma\in \Omega^2(M)$, cannot integrate to 
a Lie groupoid unless the subgroup of spherical periods 
\[ \Lambda= \big\{\int_{S^2}f^*\sigma\big|\  f\colon S^2\to M\big\}\subset \R\] 
is discrete.  Using a \v{C}ech-theoretic argument, Mackenzie \cite{mac:gen} generalized this obstruction to more general transitive Lie algebroids (that is, having a surjective anchor map $\a\colon A\to TM$). For cotangent Lie algebroids of Poisson manifolds $(M,\pi)$, Catteneo-Felder \cite{cat:poi} used an infinite-dimensional symplectic quotient to 
construct a topological groupoid which, when smooth, would be a symplectic groupoid integrating $T^*M\Ra M$.  
In his talk at the conference `Poisson 2000' in Luminy, \v{S}evera \cite{sev:so} proposed a topological groupoid $G\rra M$ associated to arbitrary Lie algebroids $A\Ra M$, as the space of Lie algebroid paths $\gamma\colon \sI\to A$ modulo Lie algebroid homotopies, but did not make precise under what conditions $G$ is smooth. Independently, a similar idea was suggested by Weinstein  (see comments in \cite{cra:lect}), motivated by the new proof of Lie's Third Theorem by Duistermaat-Kolk \cite{du:li}. The full picture was obtained by Crainic-Fernandes  \cite{cra:intlie}, who gave a careful discussion of the smoothness properties of the topological groupoid $G\rra M$, introduced the so-called \emph{monodromy groups}, and 
proved  that uniform discreteness of these groups is both necessary and sufficient for smoothness.

The construction of $G\rra M$ as the   space of Lie algebroid paths modulo Lie algebroid homotopy is conceptually satisfying, but can be hard to work with in practice, due to technical difficulties with  infinite-dimensional manifolds, and since the concept of Lie algebroid homotopy is somewhat involved.  It is the purpose of this note to show that in the special case of transitive Lie algebroids, a rather elementary proof of the Crainic-Fernandes  theorem is available. 
Our construction of an integration is quite direct, and is more geometric than that given in Mackenzie's work \cite{mac:gen}. The basic idea is to adapt the re-construction of a principal bundle with connection  from its parallel transport (holonomies) \cite{bar:hol,cae:fam,col:par,sch:par}. 

Let us give an overview of the construction. We start with a geometric interpretation of the Crainic-Fernandes monodromy groups. Recall that the clutching construction identifies the isomorphism classes of principal $K$-bundles over 2-spheres with elements of the fundamental group $\pi_1(K,e)$. We show that, similarly, transitive Lie algebroids over $S^2$, with structure Lie algebra $\k$, 
are classified by elements of  $\on{Cent}(\wt{K})$, 
the center of the connected, simply connected Lie group  integrating $\k$. Given a transitive Lie algebroid $A\Ra M$, 
with isotropy algebras $\h_m,\ m\in M$, we obtain the monodromy groups $\Lambda_m\subset \on{Cent}(\wt{H}_m)$ as the set of elements classifying the pullback Lie algebroids $f^!A\Ra S^2$, for all maps $f\colon S^2\to M$ taking the base point 
$m_0\in S^2$ 
to $m$. The same argument as  in \cite{alm:sui} shows that discreteness of $\Lambda_m$ is necessary for existence of an integration. 
Conversely, if the monodromy groups are  discrete, a quotient Lie group bundle \[ \mathsf{U}=\wt{H}/\Lambda=\bigcup_{m\in M} \wt{H}_m/\Lambda_m\] is well-defined. A splitting $j\colon TM\to A$ 
(right inverse to the anchor) defines a connection on this Lie group bundle. Equivalently,  via parallel transport, this gives an action of the \emph{path groupoid} 
 $\on{Path}(M)\rra M$ (see Section \ref{subsec:pathgroupoid})  on $\mathsf{U}$, compatible with the   
group multiplication, and hence a 
semi-direct product groupoid 
\[ \on{Path}(M)\ltimes \mathsf{U}\rra M.\]
The  isotropy groups of the path groupoid $\on{Path}(M)$ are thin homotopy classes 
of loops in $M$; they assemble into a group bundle  $\on{Loop}(M)$. Let $\on{Loop}^0(M)$  be the normal 
subgroup bundle  of \emph{contractible loops}. Then $\Pi(M)= \on{Path}(M)/\on{Loop}^0(M)\rra M$ is the \emph{fundamental groupoid}, consisting of homotopy classes of paths. 
We will describe a canonical group bundle morphism $\on{Loop}^0(M)\to \mathsf{U}$, so that $\on{Loop}^0(M)$ becomes a normal subgroupoid of the semi-direct product $\on{Path}(M)\ltimes \mathsf{U}$. 
\medskip

\noindent{\bf Theorem.} The 
	quotient 
\[ 	G=(\on{Path}(M)\ltimes \mathsf{U})/\on{Loop}^0(M)\rra M\]
is the source-simply connected Lie groupoid integrating $A\Ra M$.	
\medskip

The description of the manifold structure on $G$ in the general case is analogous to that of the fundamental groupoid, and makes $G$ into a fiber bundle over  $\Pi(M)$ with fibers $\mathsf{U}_m$. 

The method extends to \emph{regular} Lie algebroids. By definition, the anchor map of a regular Lie algebroid 
has constant rank, and so defines a regular foliation $\ca{F}$ of $M$. Here, one needs the stronger condition of \emph{locally uniform discreteness} of the monodromy groups. With these assumptions, the source-simply connected Lie groupoid integrating $A$ has a similar description as above, replacing the path groupoid with a groupoid of thin homotopy classes of $\ca{F}$-paths, and $\on{Loop}^0(M)$ with the corresponding bundle of contractible $\ca{F}$-loops. 
Here $G\rra M$ is a  non-Hausdorff Lie groupoid, in general.

Note that for general Lie algebroids, there is an elegant construction of integrations to \emph{local Lie groupoids} using Lie algebroid sprays, due to  Cabrera-Marcut-Salzar
\cite{cab:loc} (generalizing an approach of Crainic-Marcut \cite{cra:exi} for the Poisson case). This construction, providing integrations `near the units', is in some sense orthogonal to the one given here.  It would be interesting to find a way of combining the two constructions. 

 \medskip
\noindent  {\bf Acknowledgement.} It is a pleasure to thank Henrique Bursztyn, as well as the referee, for a number of helpful comments.

 \section{Reminders on Lie groupoids and Lie algebroids}\label{sec:basics}
 We review the basic definitions for Lie algebroids and Lie groupoids; for details see, e.g., 
  \cite{cra:lect,duf:po,mac:gen}.
 \subsection{Lie groupoids}
 A Lie groupoid is a manifold $G$ of  \emph{arrows}, together with a submanifold $M\subset G$ of \emph{units}, surjective submersions $\sz,\tz\colon G\to M$ called \emph{source} and \emph{target}, and a smooth multiplication map 
 from the submanifold 
  $G^{(2)}=\{(g',g)\in G^2|\ \sz(g')=\tz(g)\}$ of \emph{2-arrows} to $G$, denoted $(g',g)\mapsto g'g\equiv g'\circ g$, 
 with $\sz(g'g)=\sz(g),\ \tz(g'g)=\tz(g')$.   
 We picture elements $g\in G$ as arrows from the source $\sz(g)=m$ to the target $\tz(g)=m'$, 
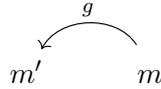
 \[\begin{tikzcd}  m' & m^{\phantom{'}}\arrow[l,bend angle=60,bend right, "g"']\end{tikzcd}\]
 and the groupoid multiplication as a concatenation of arrows,  
 
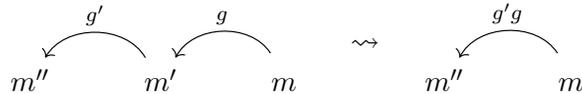
\[\begin{tikzcd} m''& m' \arrow[l,bend angle=60,bend right, "g' "']& m^{\phantom{'}}\arrow[l,bend angle=60,bend right, "g"']\end{tikzcd}
 \ \ \begin{tikzcd} \leadsto \\ \phantom{M} \end{tikzcd}\ \ \begin{tikzcd}  m'' & m^{\phantom{'}}\arrow[l,bend angle=60,bend right, "g'g"']\end{tikzcd}
 \]
The groupoid multiplication is required to satisfy axioms of \emph{associativity} $(g''g')g=g''(g'g)$, 
\emph{units} $g\circ \sz(g)=g=\tz(g)\circ g$, and existence of \emph{inverses} $g^{-1}$, that is,  $g\circ g^{-1}=\tz(g)$ and $g^{-1}\circ g=\sz(g)$. Applications to foliation theory dictate that one should allow for the total space of $G$ to be non-Hausdorff, but one requires that the space of units $M$ as well as all source and target fibers are Hausdorff.  
\begin{examples}
	\begin{enumerate}
		\item  Lie groupoids with unit space $M=\on{pt}$ are just Lie groups $K\rra \pt$. More generally, 
		Lie groupoids with $\sz=\tz$ are smooth families of Lie groups $\bigcup_{m\in M} K_m$. (The union can be 
		non-Hausdorff.) 
		Note that the Lie groups $K_m$  need not be isomorphic (or even diffeomorphic), hence the family of Lie groups  is not necessarily a `Lie group bundle'. 	
		\item The \emph{pair groupoid} $\on{Pair}(M)\rra M$ consists of pairs $(m',m)\in M\times M$. The source and target maps are  
		$\sz(m',m)=m,\ \tz(m',m)=m'$, and the multiplication reads as $(m'',m')\circ (m',m)=(m'',m)$. Closely related is the \emph{homotopy groupoid} $\Pi(M)\rra M$, consisting of homotopy classes of paths $\gamma\colon \sI\to M$ from $\sz([\gamma])=\gamma(0)$ to $\tz([\gamma])=\gamma(1)$. 
		\item Given a Lie group $K$ with an action on a manifold $M$, one has the \emph{action groupoid} $G=K\times M\rra M$, where $(k,m)$ is viewed as an arrow from $m$ to $k\cdot m$. 
		\item  A foliation $\ca{F}$ of a manifold defines a \emph{monodromy groupoid}, consisting of homotopy classes of paths in leaves, and a \emph{holonomy groupoid} consisting of holonomy classes of such paths (see e.g. \cite{moe:fol}). 	For interesting foliations, these are often non-Hausdorff. 
	\end{enumerate}
\end{examples}
Morphisms of Lie groupoids are smooth maps preserving all the structures. For example, the inclusion of units 
$M\to G$ is a morphism, and the 
\emph{Lie groupoid anchor} 
\[ \a_G=(\tz,\sz)\colon G\to \on{Pair}(M)\]
is a morphism. 

We shall need some basic constructions with Lie groupoids: products, pull-backs, and quotients. Products $G'\times G\rra M'\times M$ of Lie groupoids are defined in an obvious way.  
Given a Lie groupoid $G\rra M$ and a smooth map $\varphi\colon Q\to M$ such that the induced map of pair groupoids 
$\on{Pair}(Q)\to \on{Pair}(M)$
is transverse 
to $\a_G\colon G\to \on{Pair}(M)$, one  has a \emph{pullback Lie groupoid}
 \[ \varphi^!G\rra Q\] 
 given as a fiber product 
 $\on{Pair}(Q)\times_{\on{Pair}(M)}G$, with groupoid structure as a subgroupoid of $\on{Pair}(Q)\times G$. 
 (More generally, one can replace transversality with a clean intersection condition.) 
 For quotients of $G\rra M$, we shall only consider subgroupoids that are \emph{wide} in the sense that their space of 
 units is all of $M$. Given a  wide closed Lie subgroupoid $L\rra M$, contained in $\ker(\a_G)$ and \emph{normal}
 in the sense that $glg^{-1}\in L$ for all composable $g\in G,\ l\in L$, the 
 \emph{quotient groupoid} $G/L\rra M$ is the set of equivalence classes under the relation $g'\sim g\ \Leftrightarrow\ 
 g'g^{-1}\in L$. 
  There are much more general quotient constructions where the space of units changes as well (see 
 Mackenzie \cite{mac:gen}), but the setting considered here is enough for our purposes.

 \subsection{Lie algebroids}
 We use the notation $A\Ra M$ for Lie algebroids; this notation (which we learned from \cite{bur:vec}) 
 is meant to suggest a `limit' of a Lie groupoid  $G\rra M$ when  source and target become `infinitesimally close'. 
 A Lie algebroid is a vector bundle $A\to M$ with a Lie bracket on its space of sections, together with a bundle map  
  \[ \a_A\colon A\to TM\]
  called the Lie algebroid anchor, such that the Leibnitz rule $[\sigma,f\tau]=f[\sigma,\tau]+(\L_{\a(\sigma)}f)\,\tau$ holds for all sections $\sigma,\tau$ and all functions $f\in C^\infty(M)$.   
  
 \begin{examples}
 	\begin{enumerate}
 		\item Lie algebroids over $M=\pt$ are Lie algebras $\k$. More generally, a Lie algebroid with zero anchor is a smooth family of Lie algebras $\bigcup_{m\in M}\k_m$.
 		\item The tangent bundle $TM$ is a Lie algebroid, with anchor the identity.  
 		\item Given a Lie algebra action of $\k$ on $M$, one has the action Lie algebroid $\k\times M\Ra  M$, with 
 		anchor the bundle map to $TM$ given by the infinitesimal action. 
 		\item Given a foliation $\ca{F}$ of a manifold, one has the tangent bundle $T_\F M$ to the foliation. 
 	\end{enumerate}
 \end{examples} 
   Given two Lie algebroids $A'\Ra M',\ A\Ra M$, there is a unique \emph{product Lie algebroid} 
 $A'\times A\Ra M'\times M$ such that the map $\Gamma(A')\oplus \Gamma(A)\to \Gamma(A'\times A),\ 
 (\sigma',\sigma)\mapsto \pr_1^*\sigma'+\pr_2^*\sigma$ preserves Lie brackets.  
   If $A\Ra M$ is a Lie algebroid, and $\varphi\colon Q\to M$ is a smooth map whose tangent map is transverse to the anchor $\a_A\colon A\to TM$, one has a \emph{pull-back Lie algebroid} 
 \[ \varphi^!A\Ra Q\] 
 given as the fiber product $TQ\times_{TM} A$;  the Lie algebroid structure comes from its inclusion as a subalgebroid of $TQ\times A$. Given a wide subbundle $\mf{l}\subset A$, contained in $\ker(\a_A)$ and such that the space of sections is an ideal in $\Gamma(A)$, one obtains a \emph{quotient Lie algebroid} $A/\mf{l}$. 
 Morphisms of Lie algebroids are vector bundle maps $\varphi\colon A'\to A$ whose graph is a Lie subalgebroid  
 $\on{Graph}(\varphi)\subset A\times A'$. In particular, the anchor $\a_A\colon A\to TM$ is a morphism of Lie algebroids.

 \subsection{Lie functor}
 The Lie functor takes Lie groupoids $G\rra M$ to Lie algebroids $\on{Lie}(G)\Ra M$, and morphisms of Lie groupoids to morphisms of Lie algebroids.   As a vector bundle, $\on{Lie}(G)$ is the normal bundle of $M$ inside $G$,
 \[\on{Lie}(G)=\nu(G,M),\]  The bracket on sections is induced from the identification of $\Gamma(A)$  with left-invariant vector fields on $G$ (tangent to $\tz$-fibers). 
 The Lie algebroid anchor $\a_{\on{Lie}(G)}$ is obtained by applying the normal functor to the Lie groupoid anchor $\a_G\colon G\to \on{Pair}(M)$, and identifying $\nu(\on{Pair}(M),M)=TM$ by projection to the first factor.  For the 
  pair groupoid, this gives $\on{Lie}(\on{Pair}(M))=TM$ with the standard Lie 
 bracket of vector fields.  Given a morphism of Lie groupoids $G_1\to G_2$, the corresponding morphism of Lie algebroids is again given by the normal bundle functor to the map of pairs $(G_1,M_1)\to (G_2,M_2)$.

 \begin{definition}
 A Lie algebroid $A\Ra M$ is called \emph{integrable} if there exists a Lie groupoid $G\rra M$ (possibly non-Hausdorff) such that $A\cong \on{Lie}(G)$. 
 \end{definition}If an integration exists, then it may be chosen to have connected and simply connected  source fibers. As a consequence of Lie's Second Theorem for Lie groupoids  (proved by Mackenzie-Xu,\cite[Theorem A.1]{mac:int}) such an integration is then unique. See \cite{cra:lect} for further background on the integration problem, and 
 \cite{hoy:hau} for a detailed discussion of Hausdorff aspects.
 \begin{enumerate}
 	\item Finite-dimensional Lie algebras are integrable by Lie's Third Theorem. The generalization to smooth families of Lie algebras is due to Douady-Lazard \cite{dou:esp}. The paper \cite{dou:esp} also gives an example showing that the integration of such a smooth family may fail to be Hausdorff; see Example \ref{ex:douady} below. 
 	\item Both $\on{Pair}(M)$ and $\Pi(M)$ integrate $TM\Ra M$.
 	\item  	The action Lie algebroid $\k\times M\Ra  M$ for an action of a Lie algebra $\k$ is always integrable,  
 	as shown by Dazord \cite{daz:int}. Note that the integration need not be an action groupoid for a Lie group action, due to possibly incomplete generating vector fields.  
 	\item The monodromy groupoid and the holonomy groupoid of a foliation $\F$ are both integrations of 
 	 the tangent bundle $T_\F M\Ra M$ to the foliation. 
 \end{enumerate}
 For any Lie algebroid $A\Ra M$, the anchor map defines a singular foliation given by the submodule of the $C^\infty(M)$-module of vector fields, $\a(\Gamma(A))\subset \mf{X}(M)$. The leaves of this singular foliation 
 are  called the \emph{orbits} of $A\Ra M$. 
 For any orbit $i\colon \O\to M$, the restriction (i.e, 
 pullback as a vector bundle) is  a well-defined Lie algebroid $A_\O\Ra \O$, and $A=\sqcup_\O A_\O$. 
 
 If $A\Ra M$ integrates to a source-connected Lie groupoid $G\rra M$, then  the leaves of this singular foliation coincide with the \emph{orbits}  $\O=\tz(\sz^{-1}(m))$ of $G$, and $G_\O:=\tz^{-1}(\O)\cap \sz^{-1}(\O)$ is a Lie groupoid integrating $A_\O$. The full Lie groupoid is the union $G=\sqcup_\O G_\O$. 
  To get a feeling for the integration problem, one might thus start with Lie algebroids and Lie groupoids having a  single orbit; this is the \emph{transitive case}.

 \section{Transitive Lie groupoids and Lie algebroids}
 \subsection{Transitive Lie groupoids}
 A Lie groupoid $G\rra M$ is called \emph{transitive} if it has a unique orbit: for any two elements $m,m'\in M$ there is an arrow $g$ from $m=\sz(g)$ to $m'=\tz(g)$. In other words, 
 the groupoid anchor $\a_G= (\tz,\sz)\colon G\to \on{Pair}(M)$ is surjective.  The anchor of a transitive Lie groupoid is automatically a submersion, and its 
  kernel $H=\ker(\a_G)$, given as  the pre-image of the units $M\subset \on{Pair}(M)$, is the bundle of \emph{isotropy Lie groups} 
 \[ H_m=\tz^{-1}(m)\cap \sz^{-1}(m);\]
 it fits into an exact sequence of Lie groupoids
 \[ 1\to H\to G\to \on{Pair}(M)\to 1.\]
 %
 Products, pullbacks, and quotients of transitive groupoids are again transitive. (Note that pullbacks of transitive Lie groupoids are defined for arbitrary smooth maps, since the transversality condition is automatically satisfied.) 
  
 \begin{example}[Gauge groupoids]
 Any  principal $K$-bundle $\pi\colon P\to M$, with action map $K\times P\to P,\ (k,p)\mapsto k\cdot p$, 
 defines a transitive Lie groupoid called the \emph{gauge groupoid} 
  \begin{equation}\label{eq:gaugegroupoid} G(P)=\on{Pair}(P)/K\rra M;\end{equation}
 here the quotient is by the diagonal $K$-action.  The pullback of a gauge groupoid under a smooth map $\varphi\colon Q\to M$ is given by  $\varphi^!G(P)=G(\varphi^* P)$, with the usual pullback of principal bundles $\varphi^*P=Q\times_{M} P$. 
 \end{example}
 
 It is an important fact that \emph{every} transitive Lie groupoid $G\rra M$ is isomorphic to a gauge groupoid. To see this, pick a base point $m_0$, and let $K=H_{m_0}$ be the isotropy group at $m_0$. 
Then 
\[ P=\sz^{-1}(m_0)\stackrel{\tz|_P}{\lra} M\]
is a principal $K$-bundle for the action given by $k\cdot p=p\circ k^{-1}$, and the map 
$\on{Pair}(P)\to G,\ (p',p)\mapsto p'\circ p^{-1}$ descends to an isomorphism $G(P)\to G$. 

\begin{remark}
Note that this identification uses the choice of a base point $m_0$, and that the corresponding $P$ comes with a trivialization  at $m_0$.	Given another base point $m_0'$, the choice of an arrow $g\in G$ from $m_0=\sz(g)$ to $m_0'=\tz(g)$ gives an isomorphism between the isotropy groups $K,K'$ at $m_0,m_0'$, and an isomorphism of principal bundles $P,P'$ given as the respective source fibers.
\end{remark}


\subsection{Transitive Lie algebroids}
A Lie algebroid $A\Rightarrow M$ is \emph{transitive} if its anchor map $\a_A\colon A\to TM$ is surjective. 
The kernel of the anchor map is the bundle 
of isotropy Lie algebras $\h_m=\ker(\a_A|_m)$; it fits into 
an exact sequence of Lie algebroids 
\begin{equation}\label{eq:exactla} 0\to  \h \to A\to TM\to 0.\end{equation}
Pullbacks of transitive Lie algebroids are defined for arbitrary smooth maps $\varphi\colon Q\to M$; the transversality condition is automatic. The pullback $\varphi^!A\Ra Q$ is again transitive, with 
$\ker(\a_{\varphi^!A})=\varphi^*\h$. 

\begin{example}[Atiyah algebroid \cite{at:com}]
	For any  principal $K$-bundle $\pi\colon P\to M$, the \emph{Atiyah algebroid}
	\begin{equation}\label{eq:atiyahalgebroid} A(P)=TP/K\Ra M\end{equation} 
	is a transitive Lie algebroid. 
	The anchor map $\a\colon A(P)\to TM$  is induced by the map $T\pi\colon TP\to TM$, and the bracket on sections 
	comes from their identification with $K$-invariant vector fields on $P$. 
	Under pullbacks, $\varphi^!A(P)=A(\varphi^*P)$. The identification $TP=\on{Lie}(\on{Pair}(P))$ descends to an isomorphism  $A(P)=\on{Lie}(G(P))$. 
\end{example}

A vector  bundle splitting $j\colon TM\to A(P)$ of the anchor map is equivalent to a $K$-equivariant splitting of 
the bundle map $TP\to \pi^* TM$, and is hence  equivalent to 
a principal connection on $P$. 
For this reason, vector bundle splittings   $j\colon TM\to A$ of arbitrary transitive Lie algebroids are sometimes called connections \cite[Section 17.2]{ca:ge}. (We will not follow this convention, to avoid confusion with vector bundle connections on $A$.) The curvature $R\in \Omega^2(M,\h)$ of a splitting is the $\h$-valued 2-form defined by 
\begin{equation}\label{eq:r} R(X,Y)=[j(X),j(Y)]-j([X,Y])\end{equation}
for $X,Y\in\mf{X}(M)$.  Given a smooth map $\varphi\colon Q\to M$ there is a natural pull-back splitting on $\varphi^!A$ given by \[
TQ\to \varphi^!A=TQ\times_{TM} A,\ \ 
 v\mapsto \big(v,j(T\varphi(v))\big),\] 
with curvature $\varphi^*R\in \Omega^2(Q,\varphi^*\h)$. 

\begin{lemma}\label{lem:conn}\cite{cra:intlie}
A splitting $j\colon TM\to A$ of a transitive Lie algebroid $A\Ra M$ induces a linear connection $\nabla$ 
on the Lie algebra bundle $\h=\ker(\a_A)$, by 
\[ \nabla_X\sigma=[j(X),\sigma],\ \ \ \ X\in\mf{X}(M),\ \sigma\in \Gamma(\h).\]
This connection preserves the subbundle $\on{Cent}(\h)$ of centers of isotropy Lie algebras. 
The restriction to $\on{Cent}(\h)$ does not depend on the choice of $j$, and is a flat connection. 
\end{lemma}
\begin{proof}
It is clear that $\nabla_X$ is a derivation of the $C^\infty(M)$-module structure on $\h$. 	
The Jacobi identity for the Lie algebroid bracket shows that $\nabla_X$ is a derivation of the bracket on sections of $\h$. In particular, the subbundle $\on{Cent}(\h)$ is preserved.  The curvature of $\nabla$ is expressed in terms of \eqref{eq:r} as
\[ R^\nabla\!(X,Y)\,\sigma=[R(X,Y),\sigma]\]
for $X,Y\in\mf{X}(M),\ \sigma\in \Gamma(\h)$. In particular, the right hand side vanishes if $\sigma$ is a 
section of $\on{Cent}(\h)$. Finally, note that any two splittings $j,j'$ differ by a bundle map $TM\to \h$, 
hence 
\[ \nabla'_X-\nabla_X=[j'(X)-j(X),\cdot]\]
vanishes on  $\Gamma(\on{Cent}(\h))$; this shows that the restriction to $\on{Cent}(\h)$ does not depend on the choice of $j$. 
\end{proof}
\begin{lemma}
	Let $A\rra M$ be a transitive Lie groupoid, and $Q$ another manifold. If 
	\[ \varphi \colon \R\times Q\to M,\ (t,q)\mapsto \varphi_t(q)\] is a smooth map, then the Lie algebroids $\varphi_t^!A \rra Q$	are all isomorphic. 
\end{lemma}
\begin{proof}
	Choose a splitting $j\colon T(\R\times M) \to \varphi^!A$. The choice of $j$ gives a lift of the vector field $X=\f{\p}{\p t}$ to a section $\sigma=j(X)$ of 
	$\varphi^!A$. The Lie algebroid bracket $[\sigma,\cdot]$ is an infinitesimal Lie algebroid automorphism, so it corresponds to a vector field $\wt{X}$ on the total space of $\varphi^!A$, projecting to $X$, and such that the 
	flow is by Lie algebroid morphisms. This flow  gives Lie algebroid isomorphisms between the pullbacks $\varphi_t^!A$.  
\end{proof}

One consequence of this result is that the isotropy bundle $\h=\ker(\a_A)$ of a transitive Lie algebroid $A\Ra M$ is 
more than  just a Lie algebroid with zero anchor: it is a locally trivial Lie algebra bundle.  

\begin{proposition}\label{prop:triv} For any transitive Lie algebroid $A\Ra M$, the bundle of isotropy Lie algebras 
	\[ \h=\ker(\a_A)\to M\] is a locally trivial Lie algebra bundle. In fact, if $U\subset M$ is a contractible open subset, then any choice of a smooth deformation retraction of 
	$U$ onto  a base point $m_0\in U$ determines an isomorphism of Lie algebroids $A|_U\cong TU\times \h_{m_0}$. 
\end{proposition}
\begin{proof}
Let $i\colon \{m_0\}\to U$ be the inclusion, and $p\colon U\to \{m_0\}$ the projection. Let 
$\varphi\colon \R\times U\to U,\ (t,m)\mapsto \varphi_t(m)$ be a smooth map with $\varphi_0=i\circ p$ and $\varphi_1=\on{id}_U$. By the lemma, $\varphi$ determines a Lie algebroid isomorphism 
\[ A|_U=\varphi_1^!A\stackrel{\cong}{\lra} \varphi_0^!A=p^!i^!A=p^!\h_{m_0}=TU\times \h_{m_0}.\qedhere\] 
\end{proof}
In particular,  the isotropy algebras $\h_m$ of a transitive Lie algebroid are all isomorphic. We refer to any Lie algebra $\k$ in this isomorphism class as the \emph{structure Lie algebra of $A$}. An isomorphism of $A\Ra M$ with the trivial transitive Lie algebroid $TM\times \k\Ra M$ will be called a \emph{trivialization} of $A$.

\subsection{Gauge transformations}
Given a transitive Lie algebroid $A\Ra M$ with structure Lie algebra $\k$, we can consider its group 
$\on{Gau}(A)$ of \emph{gauge transformations}, consisting of Lie algebroid automorphisms of $A$ preserving the 
anchor map. If $A=A(P)$, for a principal $K$-bundle $\pi\colon P\to M$, then every gauge transformation of the principal bundle $P$ determines a gauge transformation of $A$.
We shall need a description of the gauge group for a trivial transitive Lie algebroid 
\[ A=TM\times\k\Ra M.\] 
The group of 
Lie algebra automorphisms $\Psi\in \on{Aut}(\k)$ acts by gauge transformations 
\begin{equation}\label{eq:psiaction} 
v+\xi\mapsto v+\Psi(\xi),\end{equation} 
while the group of functions 
$f\in C^\infty(M,\wt{K})$ acts by gauge transformations 
\begin{equation}\label{eq:faction} v+\xi\mapsto v+\Ad_f\xi-\iota_v f^*\theta^R;\end{equation}
here $\wt{K}$ denotes the connected, simply connected Lie group integrating $\k$, and $\theta^R
\in \Omega^1(\wt{K},\k)$ is the right-invariant Maurer-Cartan form. 

Hence, we obtain an action of the  semi-direct product of these two groups.

\begin{proposition}\label{prop:gauge}
If $M$ is connected and simply connected, then the map 
\[ C^\infty(M,\wt{K})\rtimes \on{Aut}(\k)\to \on{Gau}(TM\times \k)\] 
is surjective. Its kernel consists of pairs $(c^{-1},\Ad_{c})$ with $c\in \wt{K}$ (as a constant function).  
Equivalently, given a base point $m_0$, every gauge transformation is given by a unique pair $(f,\Psi)$ such that $f(m_0)=e$. 
\end{proposition}

\begin{proof}
		The group of vector bundle automorphism of $TM\times\k\to M$, preserving the projection to $TM$ and inducing the identity on the base, is a semi-direct product 
	$\Omega^1(M,\k)\rtimes C^\infty(M,\GL(\k))$.
	Elements $(\theta,\Phi)$ of this group act on sections $X+\xi$, with $X\in\mf{X}(M),\ \xi\in C^\infty(M,\k)$,  by 
	\[ (\theta,\Phi)\cdot (X+\xi)=X+\iota_X\theta+\Phi(\xi).\]	
	Such a transformation preserves the Lie algebroid bracket 
	\[ [X+\xi,Y+\eta]=[X,Y]+[\xi,\eta]+\L_X \eta-\L_Y\xi\]  if and only if the following three equations are satisfied, for all $X,Y\in\mf{X}(M),\ \xi,\eta\in C^\infty(M,\k)$:
	\begin{align*}
	\Phi([\xi,\eta])&=[\Phi(\xi),\Phi(\eta)],\\
	\theta([X,Y])&=\L_X\iota_Y\theta-\L_Y\iota_X\theta+[\iota_X\theta,\iota_Y\theta],\\\ \Phi(\L_X\eta)&=\L_X(\Phi(\eta))+[\iota_X\theta,\Phi(\eta)].
	\end{align*}
	The first condition says that $\Phi$ takes values in $\on{Aut}(\k)$, 
	the second condition says that $\theta$ satisfies the Maurer-Cartan equation $\d\theta+\hh[\theta,\theta]=0$, and the  third equation means $\L_X\Phi+\ad_{\theta(X)}\circ \Phi=0$ for all $X$, i.e.,
	$\d\Phi+\ad_\theta\circ \Phi=0$. This description of gauge transformations of $TM\times \k\Ra M$ in terms of pairs $(\theta,\Phi)$ is due to Mackenzie  \cite[Section 8.2]{mac:gen}, and does not require $M$ to be connected and simply connected.  
		We shall use these conditions now. Fix a base point $m_0\in M$. 
	Since $\theta$ satisfies the Maurer-Cartan equation, and $M$ is simply connected, there is a unique function 
	$f\in C^\infty({M},\wt{K})$, with $f(m_0)=e$, 
	such that 
	\[ \theta=-f^*\theta^R.\] 
	Next, 
	\[ \d {\Phi} + \ad_{{\theta}}\circ {\Phi}=0 \Rightarrow 
	\d (\Ad_{{f}^{-1}} \circ {\Phi})=0 \Rightarrow \Ad_{{f}^{-1}}\circ\, {\Phi}=\on{id}
	\Rightarrow {\Phi}=\Ad_{{f}}\circ \Psi
	\]
	where $\Psi=\Phi(m_0)\in \on{Aut}(\k)$. We hence see that the action of $(\Phi,\theta)$ is the action of 
	$f$ followed by the action of $\Psi$. This action is trivial if and only if $f^*\theta^R=0$, which means that $f=c^{-1}$ is a constant function, and $\Ad_{{f}}\circ \Psi=\on{id}$, so $\Psi=\Ad_c$. 
\end{proof}

\subsection{Framings}\label{subsec:framings}
Let $M$ be a manifold with a base point $m_0$. A framing of a principal $K$-bundle $P\to M$ at $m_0$ is a trivialization of $P$ at $m_0$, i.e., and isomorphism $P|_{m_0}\cong K$. Similarly, given a transitive Lie algebroid 
  $A\Ra M$ be a transitive Lie algebroid, with anchor $\a$ and structure Lie algebra $\k$, we define a framing 
  at $m_0$ to be an isomorphism of Lie algebras $\phi_0\colon \ker(\a)|_{m_0}\to \k$. 
  
\begin{definition}\label{def:framings}
	\begin{enumerate}
		\item We denote by $\on{Prin}_K(M,m_0)$ the set of isomorphism classes of principal $K$-bundles $P\to M$ 
		with framing at $m_0$.
		\item  We denote by $\on{Tran}_\k(M,m_0)$ the set of isomorphism classes of transitive Lie algebroids $A\Ra M$ with structure Lie algebra $\k$, with framing at $m_0$, modulo isomorphisms intertwining the framing. 
	\end{enumerate}
\end{definition}
A framing of $P\to M$ at $m_0$ determines a framing of $A(P)=TP/K\Ra M$ at $m_0$, giving a natural map 
\[ \on{Prin}_K(M,m_0)\to \on{Tran}_\k(M,m_0).\]

Any framing of  $P\to M$ at $m_0$ extends to a trivialization over a contractible open neighborhood of $m_0$; the extension is unique up to isotopy. Similarly, given a transitive Lie algebroid $A\ra M$ with framing at $m_0$, and a  contractible open neighborhood $U$ of $m_0$, we may choose an isomorphism of Lie algebroids $\phi\colon A|_U\to TU\times \k$, intertwining the framings. (See Proposition \ref{prop:triv}.)  By Proposition \ref{prop:gauge} the extension $\phi$ is unique up to the action of $f\in C^\infty(U,\wt{K})$ with $f(m_0)=e$; 
in particular it is unique up to isotopy.

As an application, suppose  $M',\,M$ are manifolds with base points $m_0',m_0$. After choosing a germ of a diffeomorphism between open neighborhoods of the base points, identifying the base points, the connected sum $M'\# M$ 
with base point $m_0'=m_0$ 
is defined. (As is well-known, isotopic diffeomorphisms give diffeomorphic manifolds.) Extending the framings to trivializations over open neighborhoods, as explained above, the connected sum operation extends to principal bundles and Lie algebroids. 
Once hence obtains maps 
\begin{align*}
\on{Prin}_K(M',m_0')\times \on{Prin}_K(M,m_0)\to \on{Prin}_K(M'\# M,m_0),&\ \ ([P'],[P])\mapsto [P'\# P].\\
 \on{Tran}_\k(M',m_0')\times  \on{Tran}_\k(M,m_0)\to  \on{Tran}_\k(M'\# M,m_0),&\ \ 
([A'],[A])\mapsto [A'\# A].
\end{align*}

\subsection{The path groupoid} \label{subsec:pathgroupoid}
Given a smooth path $\gamma\colon \sI\to M$ in a manifold $M$, we consider $\gamma(0)$ as the \emph{source} and $\gamma(1)$ as the \emph{target}. 
 In accordance with our conventions for groupoids, we picture paths $\gamma$ (or the intervals $\sI$) as going from the right to the left, so $0$ is the right end point and $1$ is the left end point. 
	\begin{center}
	\resizebox{3cm}{!}{
		\begin{tikzpicture}
		\begin{scope}[decoration={
			markings,
			mark=at position 0.5 with {\arrow{>}}}
		] 
		\draw[postaction={decorate}] (0,0) -- (-1,0) node[above] at (-0.5,0){\tiny $\gamma$};
		\coordinate [label=center:{\tiny $1$}]  (1) at (-1.2,0);
		\coordinate [label=center:{\tiny $0$}]  (0) at (0.2,0);
		\end{scope}
		\end{tikzpicture}
	}
\end{center}

For paths $\gamma,\gamma'\colon \sI\to M$ such that $\gamma'(0)=\gamma(1)$, the concatenation is defined as usual by 
\[ (\gamma'\ast\gamma)(t)=\begin{cases}\gamma(2t)&0\le t\le \f{1}{2}\\
\gamma'(2t-1)& \f{1}{2}\le t\le 1.\end{cases}\]
Here one encounters the problem that the concatenation of smooth paths need not be smooth. One way of dealing with this issue is to restrict attention to paths with \emph{sitting instances} \cite{mac:ho,sch:par}, i.e., paths that are constant on neighborhoods of $\{0\},\{1\}\subset \sI$. 
 A \emph{homotopy with sitting instances} $\gamma_0\simeq \gamma_1$ between two such paths  from $m$ to $m'$ is given by a smooth map 
 \begin{equation}\label{eq:homotopy}  h\colon \sI^2\to  M\end{equation}
 such that 
 \[ h(\cdot,0)=\gamma_0,\ h(\cdot,1)=\gamma_1,\ h(0,\cdot)=m,\ h(1,\cdot)=m',\] 
 and such that $h(s,t)$ depends only on $t$ for $(s,t)$ near 
 $\partial \sI\times \sI$, and depends only on $s$ for $(s,t)$ near $\sI\times \partial \sI$.  
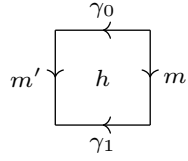
 	\begin{center}
 	\resizebox{3cm}{!}{
 		\begin{tikzpicture}
 		\begin{scope}[decoration={
 			markings,
 			mark=at position 0.5 with {\arrow{>}}}
 		] 
 		\draw[postaction={decorate}] (0,0) -- (-1,0) node[above] at (-0.5,0){\tiny $\gamma_0$};
 		\draw[postaction={decorate}] (-1,0) -- (-1,-1) node[left] at (-1,-0.5){\tiny $m'$};
 		\draw[postaction={decorate}] (0,-1) -- (-1,-1) node[right] at (0,-0.5){\tiny $m$};
 		\draw[postaction={decorate}] (0,0) -- (0,-1) node[below] at (-0.5,-1){\tiny $\gamma_1$};
 		\coordinate [label=center:{\tiny $h$}]  (h) at (-0.5,-0.5);
 		\end{scope}
 		\end{tikzpicture}
 	}
 \end{center}
 The space  of homotopy classes $[\gamma]$ is a transitive groupoid 
 \[ \Pi(M)\rra M\] 
 called the \emph{fundamental groupoid}; it is the source-simply connected integration of $TM\Ra M$.  
  Its isotropy groups are the fundamental groups $\pi_1(M,m)$. 

  One can also consider the stronger equivalence relation of \emph{thin homotopy} \cite{bar:hol,cae:fam}, denoted $\gamma_0\simeq_{\on{thin}}\gamma_1$, given by a homotopy with sitting instances such that $h$ has rank $\le 1$ everywhere. For example, reparametrizations of paths are thin homotopies. The set of thin homotopy classes 
  $[\gamma]_{\on{thin}}$ 
  is an infinite-dimensional transitive groupoid
 \[ \on{Path}(M)\rra M\]
 called the \emph{path groupoid}. The isotropy group at $m$, denoted $\on{Loop}(M)_m$, is the group of thin homotopy classes of loops based $m\in M$.   Given a connection on a vector bundle $V\to M$, one obtains an 
 action of  the path groupoid on the vector bundle (compatible with the linear structure). 
 Similarly, one can describe connections on principal bundles, Lie group bundles, Lie algebra bundles, and so on, in terms of actions of the path groupoid; flatness of the connection means that the action descends to an action of the fundamental groupoid
 \[ \Pi(M)=\on{Path}(M)/ \on{Loop}^0(M)\]
 where $\on{Loop}^0(M)\subset \on{Loop}(M)$ is the bundle of \emph{contractible loops}. 
 
 \begin{remark}
 	The path groupoid is infinite-dimensional, but it has a diffeology  for which it is a diffeological groupoid (see \cite[Appendix]{col:par}). The diffeology gives a notion of smooth maps, 
 	and 
 	a connection on a 
 	vector bundle, principal bundle, etc is equivalent, via parallel transport, to a \emph{smooth} action of  $\on{Path}(M)$.  See \cite{cae:fam,col:par,sch:par} for further discussions. 
 \end{remark}

\section{Classification of transitive Lie algebroids over 2-spheres} 
It is well-known that for Lie groups $K$, the principal $K$-bundles over 2-spheres (with fixed trivialization at some base point) 
are classified by elements of the fundamental group $\pi_1(K,e)$. We will explain that similarly, transitive Lie algebroids over $S^2$, with structure Lie algebra $\k$,  are classified  by elements of the center of $\wt{K}$.  

Choose a base point $m_0\in S^2$, and consider the set $\on{Tran}_\k(S^2,m_0)$ of isomorphism classes of transitive Lie algebroids $A\Ra S^2$ with structure Lie algebra $\k$, with framing at $m_0$, and the set $\on{Prin}_K(S^2,m_0)$ be the set of isomorphism classes of principal $K$-bundles $P\to S^2$ with framing at $m_0$. Both of these sets have group structures, with product given by the connected sum operation (see Section \ref{subsec:framings}), using $S^2\# S^2=S^2$. If $\k$ is the Lie algebra of a connected Lie group $K$, we have the group homomorphism 
\[ \on{Prin}_K(S^2,m_0)\to \on{Tran}_\k(S^2,m_0),\ \ [P]\mapsto [A(P)].\] 



\begin{theorem}\label{th:trans2}	Isomorphism classes of transitive Lie algebroids $A\Ra S^2$ with structure Lie algebra $\k$ are classified by 
	elements of the center of $\wt{K}$. In more detail:
	\begin{enumerate}
		\item 
		 There is a canonical group isomorphism, 
		\begin{equation}\label{eq:classification} \on{Tran}_\k(S^2,m_0)\stackrel{\cong}{\lra} \on{Cent}(\wt{K}),\ \ [A]\mapsto c(A).\end{equation} 
		\item If $\k$ is the Lie algebra of a connected Lie group $K$, then  the diagram 
		\begin{equation}\label{eq:commdiagram} \xymatrix{
			\on{Prin}_{K}(S^2,m_0)\ar[r]\ar[d]_\cong&  \on{Tran}_\k(S^2,m_0)\ar[d]_\cong\\
			\pi_1(K,e)\ar[r] & \on{Cent}(\wt{K}).}\end{equation}
	\end{enumerate}
commutes. 
\end{theorem}

\begin{proof}
For the proof, and much of our subsequent discussion, it is convenient to work with Lie algebroids over squares rather than spheres. By a  \emph{framing of $A\Ra \sI^2$ near 
$\partial \sI^2$}, we mean a germ of a trivialization $\phi\colon A|_U\cong TU\times \k$ on some neighborhood $U\supset \partial \sI^2$ of the boundary.  Denote by 
$\on{Tran}_\k(\sI^2,\partial \sI^2)$ the isomorphism classes of transitive Lie algebroids $A\Ra \sI^2$ with germs of framings
near $\partial \sI^2$. Then 
\[ \on{Tran}_\k(\sI^2,\partial \sI^2)=\on{Tran}_\k(S^2,m_0).\] 
To see this, choose any smooth map $\sI^2\to S^2$ taking the boundary $\partial \sI^2$ to the base point $m_0$, and inducing an orientation preserving diffeomorphism $\sI^2-\partial \sI^2\to S^2-\{m_0\}$. Such a map is unique up to isotopy, and identifies Lie algebroids over $\sI^2$ with framing near $\partial \sI^2$ with Lie algebroids over $S^2$ with framing near 
$m_0$. By the discussion in Section \ref{subsec:framings}, isomorphism classes of the latter may be used to define 
$\on{Tran}_\k(S^2,m_0)$.  
 
 Thanks to the framing, there are no special difficulties in dealing with a manifold with corners such as $\sI^2$. Working with squares makes the connected sum operation even simpler: $A'\# A$ is defined by placing the squares for $A',A$ next to each  other (horizontally), and concatenating. 

 Let $A\Ra \sI^2$ be a transitive Lie algebroid with a framing
 near $\partial \sI^2$. It represents a trivial element of $\on{Tran}_\k(\sI^2,\partial \sI^2)$ if and only if there exists a global 
 trivialization $\psi\colon A\cong T\sI^2\times \k$, extending the given framing near $\partial \sI^2$. 
In general, such a global trivialization does not exist. However,  we may always choose trivializations  
 \[ \psi_+,\ \psi_-\colon A\stackrel{\cong}{\lra} T\sI^2\times \k,\] 
 where $\psi_+$ extends the framing of $A$ on some neighborhood of the $\sqcup$ part of the boundary, 
 given as the union $(\{0\}\times \sI)\cup (\sI\times \{1\})\cup (\{1\}\times \sI) $,
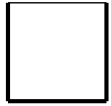
 	\begin{center}
 	\resizebox{1.7cm}{!}{
 		\begin{tikzpicture}
 		\draw (0,0) -- (-1,0) ;
 		\draw[very thick] (-1,0) -- (-1,-1) ;
 		\draw[very thick] (0,-1) -- (-1,-1) ;
 		\draw[very thick]  (0,0) -- (0,-1) ;
 		\end{tikzpicture}
 	}
 \end{center}
(but $\psi_+$ may be different from the given framing near $\sI\times \{0\}$)
 while $\psi_-$ coincides with the framing  near 
the $\sqcap$ part of the boundary, $(\{0\}\times \sI)\cup (\sI\times \{0\})\cup (\{1\}\times \sI) $,
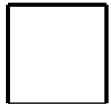
	\begin{center}
	\resizebox{1.7cm}{!}{
		\begin{tikzpicture}
		\draw[very thick] (0,0) -- (-1,0) ;
		\draw[very thick]  (-1,0) -- (-1,-1) ;
		\draw(0,-1) -- (-1,-1) ;
		\draw[very thick]  (0,0) -- (0,-1) ;
		\end{tikzpicture}
	}
\end{center}
(but $\psi_-$ may be different from the given framing near $\sI\times \{1\}$).

The two trivializations $\psi_\pm$  are related by a gauge transformation of $T\sI^2\times \k$, which is trivial near the two vertical sides 
$ \{0\}\times \sI$ and $\{1\}\times \sI$. Proposition \ref{prop:gauge} shows that this gauge transformation is given by a function $f\in C^\infty(I^2,\wt{K})$, which 
is constant with values in $\on{Cent}(\wt{K})$ near the two vertical sides (so that $\Ad_f=1$ there). 
We let 
\begin{equation}\label{eq:def}
	 c(A)=f\big|_{\{1\}\times \sI}
\ (f\big|_{\{0\}\times \sI})^{-1}
\end{equation}
The definition is independent of the choices of  $\psi_\pm$. Indeed, any other choice is obtained  by composition with a gauge transformation by $g_\pm \in C^\infty(\sI^2,\wt{K})$ , equal to $e$ on some neighborhood of the 
$\sqcup$, respectively $\sqcap$ parts of the boundary.  But this does not affect the values of $f$ at the vertical sides.  Similarly, $c(A)$ depends only on the equivalence class of the framed Lie algebroid $A$, since we may carry along the
trivializations $\psi_\pm$ with any isomorphism of framed Lie algebroids. This defines the map 
\[\on{Tran}_\k(S^2,m_0)\to \on{Cent}(\wt{K}),\ [A]\mapsto c(A).\]
We now check its properties, using at various junctures that the function $f$ is determined only up to multiplication by a constant  in  $\on{Cent}(\wt{K})$.
\medskip

{\bf Group homomorphism.} Given $A',A\Ra \sI^2$,  the chosen trivializations
 $\psi'_\pm,\psi_\pm$,  of $A',A$ glue to  trivializations 
 $\psi'_+\#\psi_+,\ \psi'_-\#\psi_-$
 for the concatenation $A'\# A$.  We may choose the functions $f',f$ so that $f'\big|_{\{0\}\times \sI}=f\big|_{\{1\}\times \sI}$; 
for any such choice they glue to a transition function $f'\#f$.  It is then immediate from \eqref{eq:def} that 
$c(A'\#A)=c(A')c(A).$\medskip

{\bf Injectivity.} Suppose $c(A)=1$, so that $f\big|_{\{0\}\times \sI}=f\big|_{\{1\}\times \sI}$. We may take 
both values to be equal to $e$. Write $f=g_- g_+^{-1}$, where $g_+$ takes on the constant value $e$ near 
the $\sqcup$ part of the boundary, while $g_-$ takes on value $e$ near the $\sqcap$ part. Changing 
$\psi_+,\psi_-$ by $g_+,g_-$, respectively, we arrange that $f=e$, i.e. $\psi_+=\psi_-=:\psi$. 
That is, $\psi$ defines a global trivialization of $A\Ra  \sI^2$ extending the framing near  $\partial \sI^2$. We conclude that the framed Lie algebroid $A$ represents the trivial element of $\on{Tran}_\k(\sI^2,\partial\sI^2)$. \medskip

{\bf Surjectivity.} Let $c\in \on{Cent}(\wt{K})$ be given. Choose a path $k\colon \sI\to \wt{K}$, with sitting instances, from $k(0)=e$ to $k(1)=c$.  We take $A\Ra M$ to be the trivial Lie algebroid $T\sI^2\times \k$, with the trivial framing near the $\sqcup$ part of the boundary, but with the framing near the upper side $\sI\times \{0\}$ modified by gauge transformation by the function $(s,t)\mapsto k(s)$. Then $A\Ra \sI^2$ with this new framing near $\partial \sI^2$ represents $c(A)$. 
\medskip

If $\k=\on{Lie}(K)$ for a connected Lie group $K$, we can use a similar method for the classification of principal $K$-bundles $P\to \sI^2$ 
with germs of trivialization $P|_U\cong U\times K$ near $\partial \sI^2$. Choose a global trivialization 
$P\cong \sI^2\times K$ that agrees with the given one on some neighborhood of the $\sqcup$ part of the boundary. The resulting trivializations of $\iota^*P\to I$ are then related by a based gauge transformation, given by a loop $\lambda\colon \sI\to K$, and the class of this loop defines $c(P)$. Comparing with the construction for Lie algebroids,  the commutativity of the diagram in (b) is now clear.
\end{proof}

\begin{example}\label{ex:r}
If $\k=\R$, we have $\on{Cent}(\wt{K})=\R$, so
 \[ \on{Tran}_\R(S^2,m_0)\cong \R.\]
Letting $\omega$ be the standard symplectic form on $S^2$, the Lie algebroid corresponding to $\lambda \in \R$ is given by $A^{(\lambda)}=TS^2\times \R$ with the bracket,  
\[ [X+f,Y+g]=[X,Y]+\L_X g-\L_Y f+\lambda\,\omega(X,Y).\] 
\end{example}
\begin{example}
For $\k$ a semisimple Lie algebra, we have that 
\[ \on{Cent}(\wt{K})=\pi_1(K,e),\]
where $K=\wt{K}/\on{Cent}(\wt{K})$ is the adjoint form of $\wt{K}$. Consequently, the map
\[ \on{Prin}_K(S^2,m_0)\to \on{Tran}_\k(S^2,m_0)\]
is an isomorphism in this case. As a special case, we see that there is a unique 
non-trivial transitive Lie algebroid $A\Ra S^2$ with structure 
Lie algebra $\mf{su}(2)$, given as  the Atiyah algebroid of the non-trivial principal $\SO(3)$-bundle over $S^2$. 
\end{example}


\section{Integration of transitive Lie algebroids}
\subsection{The monodromy groups}
Let  $A\Ra M$ be a transitive Lie algebroid. Let $\h=\ker(\a)$, and denote by $Z=\on{Cent}(\wt{H})$ the  group bundle
with fibers 
\[ Z_m= \on{Cent}(\wt{H}_m).\]
Recall from Lemma \ref{lem:conn} that 
the Lie algebra bundle $\mf{z}=\on{Cent}(\h)$ has a canonical flat connection; this induces a flat connection on the group bundle $Z$. That is, parallel transport defines an action of the fundamental groupoid $\Pi(M)\rra M$ on $Z$, preserving the group structure. 
\begin{definition}\label{def:monodromy}
		The \emph{monodromy map} of the transitive Lie algebroid $A$ at $m\in M$ 
		is the group homomorphism  
		\[  \delta_A\colon \pi_2(M,m)\to Z_m\]
		taking the homotopy class of a smooth, base point preserving map $f\colon S^2\to M$ 
		to the class $c(f^!A)$ 
		of the transitive Lie algebroid $f^!A\Ra S^2$.
		The image $\Lambda_m\subset Z_m$ of the monodromy map is called the \emph{monodromy group at $m$}.
\end{definition}
 We denote by $\Lambda\subset Z$ the union $\Lambda=\bigcup_{m\in M}\Lambda_m$.

\begin{remark}
\begin{enumerate}
	\item For possibly non-transitive Lie algebroids  $A\Ra M$, one defines the monodromy group  at $m\in M$ to be the 
	monodromy group of the transitive Lie algebroid $A_\Sigma\Ra \Sigma$ at $m$, where 
	$\Sigma\subset M$ 
	is the leaf of the singular foliation passing through $m$. 
	\item 	The monodromy groups were introduced by Crainic-Fernandes in \cite{cra:intlie}, using the notation 
	$\wt{{N}}_m(A)$.  The description in Section 3.2 of \cite{cra:intlie} is closest to the one in Definition \ref{def:monodromy}.
\item The monodromy maps are equivariant for the action of $\Pi(M)\rra M$ on the group bundles $Z$ 
and $\bigcup_m \pi_2(M,m)$. Thus, given a path $\gamma$ from $m$ to $m'$ we obtain a commutative diagram
\[ 
\xymatrix{ \pi_2(M,m) \ar[r]^{\delta_A} \ar[d]_{\gamma_*}& \ar[d]^{\gamma_*} Z_m\\
	\pi_2(M,m')\ar[r]_{\delta_A} & Z_{m'}}
\]
In particular, $\gamma_*$ takes $\Lambda_m$ to $\Lambda_{m'}$. 
\end{enumerate}
\end{remark}

\begin{proposition}\label{prop:necessary}
		A necessary condition for the integrability of a transitive Lie algebroid $A\Ra M$ is that the monodromy groups $\Lambda_m\subset Z_m$ at some (hence all) $m\in M$ are discrete. 
\end{proposition}
\begin{proof}
	Suppose $A\Ra M$ is integrable to a Lie groupoid $G\rra M$, and fix $m\in M$. Since $G$ is necessarily transitive, 
	it is the gauge groupoid
	$G=G(P)$ of the principal $K=H_m$-bundle $P=\sz^{-1}(m)$; consequently $A=A(P)$ is the Atiyah algebroid of that principal bundle.   
	Given a smooth base-point preserving map 
	$f\colon S^2\to M$, we have that $f^!A=A(f^*P)$. By the commutative diagram in Theorem \ref{th:trans2}, the 
	element  $\delta_A([f])=c(f^!A)\in Z_m$ must lie in $\pi_1(H_m,e)\subset Z_m$. 
 This shows that 
 \[ \Lambda_m\subset \pi_1(H_m,e)\subset Z_m.\] 
In particular $\Lambda_m$ must be discrete. 
\end{proof}

\subsection{Holonomies}
Let us now assume that the necessary condition from Proposition \ref{prop:necessary} is satisfied, so that all the monodromy groups $\Lambda_m$ are discrete. Then 
\[ \Lambda\subset Z=\on{Cent}(\wt{H})\subset \wt{H}\] 
are  Lie group bundles 
over $M$. Discreteness of the monodromy groups guarantees that 
\[ \mathsf{U}=\wt{H}/\Lambda,\] 
with fibers $  \mathsf{U}_m=\wt{H}_m/\Lambda_m$, is a well-defined Lie group bundle. 
\medskip

Our construction of an integration of the  transitive Lie algebroid $A\Ra M$  will use the choice of a splitting $j\colon TM\to A$.
Recall from Lemma \ref{lem:conn} that  $j$ defines a linear connection $\nabla$ on the Lie algebra bundle $\h$, 
extending the canonical flat connection on $\mf{z}$. The resulting  connection on the group bundle $\wt{H}$ 
extends the canonical  \emph{flat} connection on $Z$, and hence descends to a connection on $\mathsf{U}$. We shall denote the parallel transport along $\gamma\colon \sI \to M$ 
on any of these bundles by $\gamma_*$. As usual, it depends only on the thin homotopy class of $\gamma$, and hence defines an action of the path groupoid $\on{Path}(M)\rra M$ on these bundles. In particular, we have a groupoid action 
\[ \on{Path}(M)\times_M \mathsf{U}\to \mathsf{U}.\]
Restricting to isotropy groups,  we obtain a morphism of group bundles 
\[ \on{Loop}(M)\to \on{Aut}(\mathsf{U}),\ \ \ [\lambda]_{\on{thin}}\mapsto \lambda_*.\]
On the subgroup  bundle $\on{Loop}^0(M)$ 
of thin homotopy classes of  contractible loops, the automorphisms
$\lambda_*$ are inner: 
\begin{proposition}\label{prop:hol}
	   Let $A\Ra M$ be a transitive Lie algebroid. Any choice of splitting $j\colon TM\to A$ determines a morphism of group bundles 
	   \[ \Hol\colon \on{Loop}^0(M)\to \mathsf{U},\ \ [\lambda]_{\on{thin}}\mapsto \Hol(\lambda)\]
	   such that 
	   :
		\begin{enumerate}
			\item For all $[\lambda]_{\on{thin}}\in  \on{Loop}^0(M)_m$, the automorphism $\lambda_*\in \on{Aut}( \mathsf{U}_m)$ is conjugation by 
			$\on{Hol}(\lambda)$. 
			\item The map $\Hol\colon \on{Loop}^0(M)\to \mathsf{U}$ is equivariant for the action of $\on{Path}(M)$ on both bundles. That is, 
			for $[\lambda]_{\on{thin}}\in  \on{Loop}^0(M)$, and 
			$[\gamma]_{\on{thin}}$ with $\gamma(0)=\lambda(0)$, 
			\[ \on{Hol}(\gamma\ast \lambda\ast \gamma^{-1})=
			\gamma_*\big(\on{Hol}(\lambda)\big).\] 
		\end{enumerate}
\end{proposition}

\begin{proof}
Given any  loop $\lambda\colon\sI\to M$ with sitting instances, based at $m$, representing an element of 	$\on{Loop}(M)_m$, the
pullback Lie algebroid $\lambda^!A\Ra \sI$ comes with a natural framing near $\partial\sI=\{0,1\}$. The choice of a trivialization  
\begin{equation}\label{eq:phitrivialization} \phi\colon \lambda^!A\cong T\sI\times \h_m,\end{equation} 
compatible with the given framing near $\partial\sI$, determines an identification of $\lambda^!A$ with the Atiyah algebroid 
of the (trivial) principal bundle $\sI\times \sU_m$. The splitting of $A\Ra M$ pulls back to a splitting of 
$\lambda^!A$, hence defines a connection on this principal bundle. We may thus consider the holonomy (parallel transport) $\Hol_\phi(\lambda)\in \sU_m$, where the subscript indicates the dependence on $\phi$. 
By construction, the
automorphism $\lambda_*\in \on{Aut}(\sU_m)$ is conjugation by this element, and under concatenation of loops, 
\[ \Hol_{\phi'*\phi}(\lambda'*\lambda)=\Hol_{\phi'}(\lambda')\Hol_\phi(\lambda).\]
Given a path $\gamma$ (with sitting instances) from $m$ to $m'$, we obtain an isomorphism 
$\gamma_*\colon \h_m\cong \h_{m'}$. The pullback splitting  of $\gamma^!A\Ra \sI$ 
defines unique trivializations $\gamma^!A\cong T\sI\times \h_m\cong T\sI\times \h_{m'}$ so that the corresponding connection 1-form is zero. By concatenating with the trivialization $\phi$ of $\lambda^! A$, this 
determines a trivialization $\ti{\phi}$ of $\ti{\lambda}^!A\Ra I$, where 
$\ti{\lambda}=\gamma*\lambda*\gamma^{-1}\colon \sI\to M$ is a loop based at $m'$, with 
\[ \Hol_{\ti{\phi}}(\gamma*\lambda*\gamma^{-1})=\gamma_*\Hol_\phi(\lambda).\]

In general, $\Hol_\phi(\lambda)\in \sU_m$ depends on the choice of the trivialization $\phi$ of $\lambda^!A\Ra I$. Suppose however that the loop $\lambda$ is contractible. Then we may restrict attention to those trivializations which extend over a deformation retraction 
of $\lambda$. We will show that $\Hol_\phi(\lambda)$ becomes independent of $\phi$ within this restricted class of trivializations. In a nutshell, this uses the fact that  any two retractions of $\lambda$ combine into a map $f\colon S^2\to M$, and $\Lambda_m=\pi_1(\mathsf{U}_m,e)$ was defined in terms of such maps.  

In more detail, if $\lambda$ is a contractible loop with sitting instances we may choose a homotopy $h\colon \sI^2\to M$, with sitting instances, between $h(\cdot,0)=\lambda$ and the constant path $h(\cdot,1)=m$. 
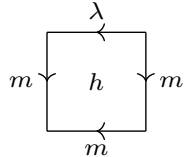
	\begin{center}
		\resizebox{3cm}{!}{
			\begin{tikzpicture}
			\begin{scope}[decoration={
				markings,
				mark=at position 0.5 with {\arrow{>}}}
			] 
			\draw[postaction={decorate}] (0,0) -- (-1,0) node[above] at (-0.5,0){\tiny $\lambda$};
			\draw[postaction={decorate}] (-1,0) -- (-1,-1) node[left] at (-1,-0.5){\tiny $m$};
			\draw[postaction={decorate}] (0,-1) -- (-1,-1) node[right] at (0,-0.5){\tiny $m$};
			\draw[postaction={decorate}] (0,0) -- (0,-1) node[below] at (-0.5,-1){\tiny $m$};
			\coordinate [label=center:{\tiny $h$}]  (h) at (-0.5,-0.5);
			\end{scope}
			\end{tikzpicture}
		}
	\end{center}
	Since $h$ takes some neighborhood  of the $\sqcup$ part of the boundary $\partial\sI^2$ to $m$, 
	the pullback Lie algebroid  $h^!A\Rightarrow \sI^2$  comes with a natural framing on this neighborhood. Choose a global trivialization 
	$\psi\colon h^!A\cong T\sI^2\times \h_m$, extending the framing on a possibly smaller neighborhood.
	This determines an identification of $h^!A$ with the Atiyah algebroid of the trivial principal bundle $\sI^2\times \mathsf{U}_m$. 	By pulling back to $\sI\times \{0\}$, we obtain the trivialization \eqref{eq:phitrivialization} of 
	$\lambda^!A$. This identifies $\Hol_\phi(\lambda)$ with the holonomy of the pull-back connection on $\sI^2\times \sU_m$.  
	
	Now let  $h'$ be another homotopy between $\lambda$ and the constant path. Gluing the homotopy $h$ with the `vertical inverse' of $h'$, we obtain a smooth map 
    \[ \sigma\colon \sI^2\to M,\ \ \sigma(s,t)=\begin{cases} h'(s,1-2t)& 0\le t\le \f{1}{2}\\
    h(s,2t-1)& \hh\le t\le 1\end{cases}\]
    \begin{center}
    	\resizebox{3.7cm}{!}{
    		\begin{tikzpicture}
    		\draw[black] (-2,0) -- (0,0) node[above] at (-1,0){\tiny $m$};
    		\draw[black] (-2,0) -- (-2,-2) node[left] at (-2,-0.5){\tiny $m$}node[left] at (-2,-1.5){\tiny $m$};
    		\draw[black] (-2,-1) -- (0,-1) node at (-0.7,-1.2){\tiny $\lambda$};
    		\draw[black] (-2,-2) -- (0,-2) node[right] at (0,-0.5){\tiny $m$}  node[right] at (0,-1.5){\tiny $m$} ;
    		\draw[black] (0,0) -- (0,-2) node[below] at (-1,-2){\tiny $m$};
    		\coordinate [label=center:{\tiny $h'$}]  (h1) at (-1,-0.45);
    		\coordinate [label=center:{\tiny $h$}]  (h2) at (-1,-1.55);
    		\end{tikzpicture}
    	}
    \end{center}
    By definition of the monodromy groups, 
    \[ c(\sigma^!A)=\delta_A([\sigma])\in \Lambda_m=\pi_1(\sU_m).\]
    This element determines an isomorphism class of framed principal 
    $\sU_m$-bundles $P\to \sI^2$ (with framing along $\partial\sI^2$) such that $c(P)=c(\sigma^!A)$, i.e. 
    $\sigma^!A\cong A(P)$. The principal bundle $P$ inherits a principal connection from the pullback splitting of $\sigma^!A$. $\Hol_\phi(\lambda)$ is interpreted as a holonomy of this pullback connection, and in particular 
    does not depend on the choice of $\phi$. 
\end{proof}

We shall also need the following relation between holonomies and the monodromy groups. 

\begin{proposition}\label{prop:realize}
	Every element of $\Lambda_m=\pi_1(\mathsf{U}_m,e)$ may be realized as a loop (with sitting instances)
	\[ \sI\to \mathsf{U}_m\ \ \tau\mapsto \Hol(\lambda_\tau),\] 
	where 	
	$\lambda_\tau\colon \sI\to M$ 
	is a smooth family of loops (with sitting instances) based at $m$, with $\lambda_0,\lambda_1$ the constant loops 
	at $m$. 
\end{proposition}

\begin{proof}
	By definition of the monodromy groups, every element of $\Lambda_m$ is realized as 	$c=\delta_A([\sigma])$
	for some smooth map 
	$\sigma\colon \sI^2\to M$, taking a neighborhood of $\partial\sI^2$ to $m$. Let $\lambda_\tau(s)=\sigma(s,\tau)$. For any fixed $\tau$, this is a loop in $M$. Using trivializations $\psi_+$ and $\psi_-$ of 
	$\sigma^!A\Ra \sI^2$, compatible with the framing on a neighborhood of the 
$\sqcup$ and $\sqcap$ 	
 part of the boundary, respectively, we obtain $\wt{H}_m$-valued holonomies related by
	\[ \Hol_{\psi_-}(\lambda_\tau)=c\,\Hol_{\psi_+}(\lambda_\tau).\]
Their images in $\mathsf{U}_m$ are $\Hol(\lambda_\tau)$, by definition. Since $\Hol_{\psi_-}(\lambda_0)=e$ and 
$\Hol_{\psi_-}(\lambda_1)=c\Hol_{\psi_+}(\lambda_1)=c$, this proves the proposition. 
\end{proof}

\subsection{Construction of an integration}
We are now in position to construct the integration of a transitive Lie algebroid $A\Ra M$ with discrete monodromy groups. The argument is motivated by the re-construction of a line bundle with connection \cite{pr:lo}, or more generally of principal bundle with connection \cite{bar:hol,mac:ho,sch:par}, from its parallel transport. 

As explained above, the splitting $j\colon TM\to A$ of $A\Ra M$ defines an action of the 
path groupoid $\on{Path}(M)\rra M$ on the group bundle $\mathsf{U}\to M$, preserving the group structure. Hence, we can form the semi-direct product groupoid 
\[ \on{Path}(M)\ltimes \mathsf{U} \rra M.\]
As a space, it consists of pairs $([\gamma]_{\on{thin}},u)$, where 
$\gamma$ is a smooth path with sitting instances, and 
$u\in \mathsf{U}_{\gamma(0)}$. The groupoid multiplication is given by 
\begin{equation} \label{eq:product}
 ([\gamma']_{\on{thin}},u')\circ ([\gamma]_{\on{thin}}  ,u)=\big([\gamma'\ast\gamma]_{\on{thin}},\ ((\gamma^{-1})_*u') u\big),\end{equation}
for paths $\gamma$ from $m$ to $m'$, $\gamma'$ from $m'$ to $m''$, and elements $u\in \mathsf{U}_{m},\ u'\in \mathsf{U}_{m'}$.  
We will define $G\rra M$ as a quotient of this groupoid by  $\on{Loop}^0(M)$, embedded diagonally as a normal subgroupoid bundle, by the map 
\begin{equation}\label{eq:embedding}
[\lambda]_{\on{thin}}\mapsto  ([\lambda]_{\on{thin}},\Hol(\lambda^{-1})\big).
\end{equation}

\begin{theorem}\label{th:explicitint}
		A transitive Lie algebroid $A\Ra M$ is integrable to a Lie groupoid if and only if the monodromy groups
		$\Lambda_m$ are discrete. In this case, the source-simply connected Lie groupoid 
		$G\rra M$ integrating $A\Ra M$ is the quotient  
		\[ G=(\on{Path}(M)\ltimes \mathsf{U})/\on{Loop}^0(M)\]
		where $\on{Loop}^0(M)$ is embedded diagonally. 
\end{theorem}
\begin{proof}
The properties of $\Hol$, as described in Proposition \ref{prop:hol}, ensure that the map \eqref{eq:embedding}
is an inclusion of the group bundle $\on{Loop}^0(M)$ as a normal subgroupoid of the semi-direct product. 
See Proposition \ref{prop:A1} in Appendix \ref{app:A} for details. Hence, the quotient 
by this normal subgroupoid is a well-defined groupoid. We have the commutative diagram,  
\[\begin{tikzcd}[column sep=huge]
	\on{Path}(M)\ltimes \mathsf{U}\arrow[r, "/ \on{Loop}^0(M) "'] \arrow[d]
	& G
	 \arrow[d] \\
	\on{Path}(M) \arrow[r, "/ \on{Loop}^0(M)"']
	& \Pi(M)
\end{tikzcd}\]

\noindent {\bf Smoothness.} 
We will give a direct description  (without reference to the diffeology of $\on{Path}(M)$) of the manifold structure on $G$. 
In fact, we will show that $G$ is a locally trivial fiber bundle over $\Pi(M)$. 
Choose a covering of $\Pi(M)$ by open subsets $\ti{O}_\nu$ on which the quotient map 
$\Pi(M)\to \on{Pair}(M)$ restricts to 
diffeomorphisms to open subsets $O_\nu\subset \on{Pair}(M)$, and in such a way that intersections 
$\ti{O}_{\nu_1}\cap \ti{O}_{\nu_2}$ map diffeomorphically onto $O_{\nu_1}\cap O_{\nu_2}$. 
By definition of the fundamental groupoid, the section $O_\nu\to \ti{O}_\nu$ may be realized 
by a smooth family of paths, i.e., by a smooth map  
\[ \sigma_\nu\colon \sI\times O_\nu\to M\]
such that each $\sigma_\nu(\cdot ,m',m)$ is a path, with sitting end points, from $m$ to $m'$, 
with homotopy class $[\sigma_\nu(\cdot,m',m)]$ contained in $\ti{O}_\nu$. The choice of $\sigma_\nu$  determines
a section 
\[ [\sigma_\nu ]_{\on{thin}}\colon O_\nu\to \on{Path}(M),\ \ (m',m)\mapsto [\sigma_\nu(\cdot,m',m)]_{\on{thin}}\]
of $\on{Path}(M)\to \on{Pair}(M)$ over $O_\nu$:
\[
\begin{tikzcd}
\on{Path}(M) \arrow[r] 
& \wt{O}_\nu \\
& O_\nu \arrow[ul,  " {[\sigma_\nu]_{\on{thin}}}  "] \arrow[u, "\cong" ,"{[\sigma_\nu]}"']
\end{tikzcd}\]
The composition of $[\sigma_\nu]_{\on{thin}}$ with the quotient map $\ti{O}_\nu\to O_\nu$
defines a section of $\on{Path}(M)\to \Pi(M)$, and  gives a local trivialization of $G$ as a group bundle over $\Pi(M)$:   
\[  F_\nu\colon O_\nu\times_M \mathsf{U}\to G|_{\wt{O}_\nu},\ (m',m,u)\mapsto \big[([\sigma_\nu(\cdot,m',m)]_{\on{thin}};u)\big].\]
Suppose $\ti{O}_{\nu_1}\cap \ti{O}_{\nu_2}$ is non-empty. Let 
 $\sigma_{\nu_2}^{-1}\ast\sigma_{\nu_1}\colon \sI\times (O_{\nu_1}\cap O_{\nu_2})\to M$
be the family of contractible loops given by concatenation of paths, 
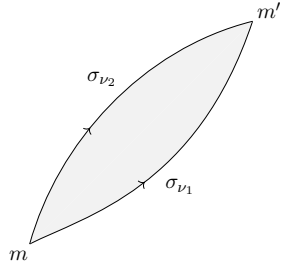
\begin{center}
	\resizebox{4cm}{!}{
		\begin{tikzpicture}	\begin{scope}[decoration={
			markings,
			mark=at position 0.4 with {\arrow{>}}}
		] 
		\draw[postaction={decorate}]  [fill=gray!10]  (-2,-2) .. controls (-1,-1.5) and (1,-1)..   (2,2)
		node[below] at (0.7,-0.7) {$\sigma_{\nu_1}$}
		;
		\draw[postaction={decorate}] [fill=gray!10]  (-2,-2) .. controls (-1.5,-0.2) and (0,1.5)  .. (2,2) 
		node[above] at (-0.7,0.7){ $\sigma_{\nu_2}$}
		;
		\coordinate [label=center:{$m$}]  (h2) at (-2.2,-2.2);
		\coordinate [label=center:{$m'$}]  (h1) at (2.3,2.2);
		\end{scope}\end{tikzpicture}
	}
\end{center}
and let 
\[ f=\Hol(\sigma_{\nu_2} ^{-1}\ast \sigma_{\nu_1})\colon O_{\nu_1}\cap O_{\nu_2}\to \mathsf{U}.\]
Then the two trivializations of $G|_{\ti{O}_{\nu_1}\cap \ti{O}_{\nu_2}}$ are related by the transition map
\[
F_{\nu_2}^{-1}\circ F_{\nu_1}\colon 
(O_{\nu_1}\cap O_{\nu_2})\times_M \mathsf{U}\to (O_{\nu_1}\cap O_{\nu_2})\times_M \mathsf{U},\ \ (m',m,u)\mapsto (m',m,\Ad_{f(m',m)}u).\]
Since $f$ is a smooth function, it follows that the transition map is smooth.  \medskip

\noindent{\bf Source-simply connectedness.}
Consider a based loop (with sitting instances)
\begin{equation}\label{eq:loopofloops} \sI\to \sz^{-1}(m)\subset G,\ \  \tau\mapsto [([\gamma_\tau]_{\on{thin}},u_\tau)]\end{equation}
in a given source fiber. It is represented by  a pair of smooth maps (with sitting instances)
\[ \gamma\colon \sI^2\to \sz^{-1}(m)\subset M,\ (\tau,t)\mapsto \gamma_\tau(t),\ \ \ u\colon \sI\to \mathsf{U}_m,\ \tau\mapsto u_\tau\]
such that 
	\begin{itemize}
		\item   $\gamma_0=\gamma_1$ the constant path at $m$,
		\item  $u_0=u_1=e$. 
	\end{itemize}

 The based loop $\sI\to \mathsf{U}_m,\ \tau\mapsto u_\tau$ defines an element  $c\in \pi_1(\mathsf{U}_m)=\Lambda_m$. By definition of the monodromy groups, this element is realized as $c=\delta_A([\sigma])$ for a smooth map $\sigma\colon \sI^2\to M$, sending a neighborhood of $\partial \sI^2$ to $m$. Consider the family of based loops in $M$, 
 $\lambda_\tau(s)=\sigma(s,\tau)$,  starting and ending with the constant loop at $m$. 
 As shown in Proposition \ref{prop:realize}, the based loop 
   \[ \sI\to \mathsf{U}_m, \tau\mapsto \Hol(\lambda_\tau)\] 
represents $c\in \pi_1(\mathsf{U}_m,e)$. 
Define $\wt{\gamma}_\tau,\ \wt{u}_\tau$ by 
	\[ u_\tau=\Hol(\lambda_\tau) \wt{u}_\tau,\ \ \ \gamma_\tau\ast \lambda_\tau=\wt{\gamma}_\tau.\]
Then $\tau\mapsto \wt{u}_\tau$ in $\mathsf{U}_m$ represents the trivial element of $\pi_1(\mathsf{U}_m,e)$, i.e., it is contractible. 
We claim 
\[  [([\gamma_\tau]_{\on{thin}},u_\tau)] = [([\ti{\gamma}_\tau]_{\on{thin}},\ti{u}_\tau)]	 .\] 
Indeed, it is clear that  $\gamma_\tau\simeq \wt{\gamma}_\tau$ since the $\lambda_\tau$ are contractible loops, while on the other hand, 
\[ \Hol(\wt{\gamma}_\tau^{-1}\ast \gamma_\tau)\ u_\tau=
\Hol(\lambda_\tau^{-1})\ u_\tau=\Hol(\lambda_\tau)^{-1}u_\tau=\ti{u}_\tau.\]
In conclusion, we can choose the representatives in \eqref{eq:loopofloops} in such a way that the loop $\tau\mapsto u_\tau$ in $\mathsf{U}_m$ is contractible. Having made such a choice, let us pick a 
smooth  homotopy $\sI^2\to \mathsf{U}_m,\ (\tau,r)\mapsto {u}_\tau^r$ between $u_\tau^1=u_\tau$ 
and the constant loop $u_\tau^0=e$, and put ${\gamma}_\tau^r(t)={\gamma}_\tau(rt)$. Then the map
\[ \sI^2\to \sz^{-1}(m)\subset G,\ (\tau,r)\mapsto 	 [([{\gamma}^r_\tau]_{\on{thin}},{u}^r_\tau)] \]
gives the desired retraction of \eqref{eq:loopofloops}.
\end{proof}

\begin{remarks}
	\begin{enumerate}
		\item From the definition of the groupoid $G\rra M$, we obtain a description of its isotropy groups as
		\[ H_m=\on{Loop}(M)_m\times_{\on{Loop}^0(M)} U_m.\]
		They fit into an exact sequence
		\[ 1\to \mathsf{U}_m\to H_m\to \pi_1(M,m)\to 1.\]
		In particular, $\mathsf{U}_m$ are the identity components of the isotropy groups. 
		\item 
		Using results about diffeologies, the proof of smoothness of $G$ can be made a bit shorter: In \cite{col:par}, it is shown that if $\ca{K}$ is a diffeological group acting on a finite-dimensional manifold $Q$, and 
		$\ca{P}\to M$ is a diffeological principal $\ca{K}$-bundle over a finite-dimensional manifold $M$, 
		then the associated bundle $\ca{P}\times_{\ca{K}}Q\to M$ is a smooth finite-dimensional fiber bundle. 
		In our context, fix a base point $m\in M$, let $\ca{K}=\on{Loop}^0(M)_{m}$, and 
		$\ca{P}=\on{Path}(M,m)$ the thin homotopy classes of paths with source $m$.  Then 
		\[ \on{Path}(M,m)\times_{\on{Loop}^0(M)_{m}} \sU_{m}\]
		is a manifold, and is a principal $H_{m}=\on{Loop}(M)_{m}\times_{\on{Loop}^0(M)_{m}} \sU_{m}$ bundle. The groupoid $G$ is its gauge groupoid. 
		\item If the monodromy groups are \emph{not} discrete, then the groups $\sU_m$ are still defined, but are not 
		Lie groups.  One way of dealing with this situation is to employ the theory of stacks, and consider $\sU_m$ as a `stacky Lie group'. It was shown by Tseng and Zhu \cite{tse:int} that the topological Lie groupoid $G\rra M$ associated to \emph{any} Lie algebroid has the structure of a `stacky Lie groupoid'. In another direction, Androulidakis-Antonini \cite{and:int} show that any transitive Lie groupoid with non-discrete monodromy groups admits a canonical lift 
		to a transitive Lie algebroid over a new manifold, in such a way that the monodromy groups do become discrete. 
	\end{enumerate}
\end{remarks}

\section{Regular Lie algebroids}
A Lie algebroid $A\Ra M$ is called \emph{regular} if the anchor map $\a_A$ has constant rank. 
Hence, the range of the anchor map defines a regular (as opposed to singular) foliation $\ca{F}$ of $M$. 
Two extreme cases of regular Lie algebroids are (i) the tangent bundle of any regular foliation $\F$, as a Lie algebroid $T_\F M\Ra M$ with anchor the inclusion to $TM$, (ii) a family of Lie algebras $\h_m$ labeled by the points of $M$, 
regarded as a Lie algebroid $\h\Ra M$ with the zero anchor. Both extreme cases admit source-simply connected integrations, however, the integrating Lie groupoids will often be non-Hausdorff.  
The source-simply connected integration of $T_\F M$ is the \emph{monodromy groupoid} $\Pi_\F(M)\rra M$.  Its arrows are homotopy classes of \emph{$\ca{F}$-paths}, that is, paths inside leaves of the foliation, modulo homotopies of such paths. Here, the non-Hausdorff phenomena arise when the foliation is not a fibration. (For example, 
it could happen that a loop $\gamma_0$ in a given leaf is not contractible, but may be approached through contractible loops $\gamma_\epsilon,\ \epsilon>0$ in nearby leaves. Then $[\gamma_\epsilon]\in M\subset T_\F M$  for $\epsilon>0$, with a limit $[\gamma_0]\not\in M$.) Similarly, a family of Lie algebras $\h_m$ integrates to the family of simply connected Lie groups $\wt{H}_m$, regarded as a Lie groupoid $\wt{H}\rra M$. 
See Douady-Lazard \cite{dou:esp} for a careful construction of the manifold structure on $\wt{H}$, as well as for the following instructive example. 

\begin{example} \label{ex:douady} \cite{dou:esp} 
Let $\h\Ra M$ be the  family of Lie algebras $\h_s\cong\R^3,\ s\in \R$ with brackets 
\[ [x,y]=sz,\ [z,x]=y,\ [z,y]=-x.\]
These Lie algebras are isomorphic to $\mf{su}(2)$ for $s>0$, to  $\R\ltimes \R^2$ (using the action by infinitesimal rotations) for $s=0$, and to $\mf{sl}(2,\R)$ for $s<0$. The corresponding simply connected Lie groups $\wt{H}_s$ are thus $\SU(2)$ for $s>0$, $\R\ltimes \R^2$ for $s=0$, and $\on{SL}(2,\R)$ for $s<0$. Their union is a groupoid $\wt{H}\rra M$ integrating 
$\h\Ra M$, but it is not Hausdorff. 
 Indeed, 
\[ \gamma\colon s\mapsto \exp_s(4\pi z)\in \wt{H}_s\] 
is a curve with the property $\gamma(s)=e_s$ (the unit in $\wt{H}_s$) for $s>0$, but $\gamma(0)\not=e_0$. 
On the other hand, it is shown in \cite{dou:esp} that there \emph{does} exist a Hausdorff integration $H\rra M$,   with 
fibers $H_s$ given by $\SO(3)$ for $s>0$, by $\on{SO}(2)\ltimes \R^2$ for $s=0$, and $\on{SO}(2,1)$ for $s<0$. 
\end{example}

For a general regular Lie algebroid $A\Ra M$, with underlying foliation $\ca{F}$, we have the exact sequence of Lie algebroids
\begin{equation}\label{eq:exact2} 0\to \h\to A\to T_\F M\to 0.\end{equation}
Recall that the monodromy group $\Lambda_m$ of $A$ at $m\in M$ is the monodromy group of the transitive Lie algebroid $A_\Sigma \Ra \Sigma$, where $\Sigma$ is the leaf of $\ca{F}$ through $m$. Let \[\Lambda=\bigcup_{m\in M}\Lambda_m\subset \wt{H}.\] 

Using the definition of the monodromy groups, one sees that $\Lambda$ is an immersed submanifold.  (Elements of $\Lambda_m$ are realized 
as $c(\varphi^! A)\in Z_m\subset \wt{H}_m$ for base point preserving maps $\varphi\colon (S^2,m_0)\to (M,m)$ taking values in leaves;  
changing $\varphi$ smoothly will lead to a smooth change of $c(\varphi^! A)$.) However, in general, $\Lambda$ need not be a closed submanifold, even if all monodromy groups are discrete. 

\begin{example}
Let $A^{(\lambda)}\Ra S^2$ be the Lie algebroid with structure Lie algebra $\R$ 
and monodromy group $\Lambda^{(\lambda)}=\lambda\Z\subset \R$, as in 
Example \ref{ex:r}. Then $A=\bigcup_\lambda A^{(\lambda)}\Ra S^2\times \R$ is a regular Lie algebroid.  We have 	
\[ \Lambda=\bigcup_{\lambda\in \R}\lambda \Z=\{(\lambda,\lambda k)|\ \lambda\in\R,\ k\in \Z\},\]
which is not a closed submanifold. 
\end{example}

\begin{definition} \cite[Section 4.1]{cra:lect}
	The monodromy groups $\Lambda_m$ of a regular Lie algebroid $A\Ra M$  are \emph{locally uniformly discrete} if there exists an open neighborhood of 
	$M$ inside $\wt{H}$, intersecting $\Lambda$ only in $M$. 
\end{definition}

\begin{remark}
For general Lie algebroids, the definition of \emph{locally uniformly discrete} is a bit more involved, since 
the union of $\wt{H}_m$ need not have a natural manifold structure. Instead, Crainic-Fernandes \cite{cra:intlie,cra:lect} 
work with the subset $\Lambda^0\subset A$ given as the union of pre-images of monodromy groups under the 
exponential maps. 
\end{remark}


\begin{proposition}
A necessary condition for integrability of a 	regular Lie algebroid
$A\Ra M$  to a (possibly non-Hausdorff) Lie groupoid $G\rra M$ is that the monodromy groups  $\Lambda_m$ are locally uniformly discrete. 
\end{proposition}
\begin{proof}
Given an integration $G\rra M$, let $H=\ker(\a_G)$ be the bundle of isotropy groups.  Since the groupoid anchor $\a_G\colon G\to \on{Pair}(M)$ has constant rank, 
this is a closed Lie subgroupoid  of  $G$. Let $\exp_{\wt{H}}\colon \h\to \wt{H}$ and $\exp_H\colon \h\to H$
be  the fiberwise exponential maps. 
 For all $m\in M$, we have that $\pi_1(H_m)\supset \Lambda_m$, hence 
 $\exp_{H_m}^{-1}(e)=\exp_{\wt{H}_m}^{-1}(\pi_1(H_m))\supset \exp_{\wt{H}_m}^{-1}(\Lambda_m)$. This shows 
 \[ \exp_{\wt{H}}^{-1}(\Lambda)\subset \exp_H^{-1}(M).\]
where $\exp_H^{-1}(M)$ is the pre-image of the unit section $M\subset H$. Let $O\subset \h$ be an open neighborhood of the zero section $M$ over which $\exp_H\colon \h\to H$ 
restricts to a diffeomorphism. Then  $M\subset \exp_{\wt{H}}^{-1}(\Lambda)\cap O\subset
\exp_H^{-1}(M)\cap O=M$. Consequently, $\exp_{\wt{H}}^{-1}(\Lambda)\cap O=M\subset \h$, and hence $\Lambda\cap \exp_{\wt{H}}(O)=M\subset \wt{H}$.
\end{proof}

By the Crainic-Fernandes theorem \cite{cra:intlie}, 
the uniform discreteness of monodromy groups is also sufficient for the existence of an integration. (Under additional assumptions, 
the integration problem for regular Lie algebroids was considered in the work of Dazord-Hector \cite{daz:int3} and Nistor \cite{nis:gro}.) 
Let us sketch how to recover this result from the construction in the previous sections, carried out leafwise. 

Suppose that the monodromy groups are uniformly discrete. Then $\Lambda\subset \wt{H}$ is a closed submanifold, and hence a closed Lie subgroupoid. Since the subgroups $\Lambda_m$ are contained 
in the center of  $\wt{H}_m$, we may take the fiberwise quotients $\mathsf{U}_m=\wt{H}_m/\Lambda_m$ resulting in a (not necessarily Hausdorff) family of Lie groups 
\[ \mathsf{U}=\wt{H}/\Lambda.\]
Consider the infinite-dimensional groupoid $\on{Path}_{\ca{F}}(M)\rra M$, consisting of thin homotopy classes of $\ca{F}$-paths. 
The choice of a splitting of the exact sequence 
\eqref{eq:exact2} defines an action of $\on{Path}_{\ca{F}}(M)\rra M$ on $\mathsf{U}$, compatible with the fiberwise group structure, hence we can form the semi-direct product
\[ \on{Path}_{\ca{F}}(M)\ltimes \mathsf{U}\rra M.\]
The bundle of isotropy groups $\on{Loop}_{\ca{F}}(M)\subset \on{Path}_\F(M)$ has a subbundle 
$\on{Loop}^0_{\ca{F}}(M)$ of contractible loops, and the source-simply connected Lie groupoid integrating $A\Ra M$ is given by the quotient 
\[ G=(\on{Path}_{\ca{F}}(M)\ltimes \mathsf{U})/\on{Loop}^0_{\ca{F}}(M).\]
To understand the manifold structure of $G$, one may proceed as in the proof of Theorem \ref{th:explicitint}, by regarding $\on{Path}_{\ca{F}}(M)$ as a principal $\on{Loop}^0_{\ca{F}}(M)$-bundle over $\Pi_\F(M)$. 
 The standard construction of the smooth structure of  $\Pi_\F(M)$ gives 
local trivializations of this principal bundle, and identifies its associated bundle $G$ 
locally with $\Pi_\F(M)\times_M \mathsf{U}$. The smooth structure of the latter determines the smooth structure on $G$, and as as in the proof of Theorem \ref{th:explicitint} one finds that it does not depend on choices. Thus, 
$G$ is a well-defined (possibly non-Hausdorff) Lie groupoid integrating $A$.

\begin{appendix}
\section{Semi-direct product of groupoids}\label{app:A}
Let $\Gamma \rra M$ be a Lie groupoid, acting on a Lie group bundle $\mathsf{U}\to M$, with action denoted 
$\gamma\cdot u$ for  $\gamma\in \Gamma,\ u\in \mathsf{U}_{\sz(\gamma)}$. We assume that the action preserves the 
structure as a Lie group bundle, that is,
\[ \gamma\cdot (u_1u_0)=(\gamma\cdot u_1)\, (\gamma\cdot u_0).\]
Then we may form a  \emph{semi-direct product} groupoid  
\[ \Gamma\ltimes \mathsf{U}\rra M.\]
As a manifold, this is the fiber product of $\Gamma$ with $\mathsf{U}$ over $M$, with respect to the source map 
$\sz\colon \Gamma\to M$. The groupoid multiplication is 
\[ (\gamma',u')\circ (\gamma,u)=(\gamma'\circ \gamma,\ (\gamma^{-1}\cdot u')u).
\]
Suppose $L\subset \ker(a_\Gamma)$ is a group subbundle of the bundle of isotropy groups, which is 
\emph{normal} in the sense that $\gamma\circ \lambda\circ \gamma^{-1}\in L_{m'}$ for all 
$\lambda\in L_m$ and all  $\gamma\in \Gamma$ with  $\sz(\gamma)=m,\ \tz(\gamma)=m'$. Then the quotient of $\Gamma$ by $L$ is a Lie groupoid
\[ \Gamma/L\rra M,\] 
in such a way that the quotient map is a morphism of Lie groupoids. 
We are interested in a lift of such a quotient to the semi-direct product. Suppose that 
we are given a morphism of group bundles 
\[ f\colon L\to \mathsf{U}\]
with the following properties: 
\begin{enumerate}
\item For every $\lambda\in L_m$ and all $u\in \mathsf{U}_m$, 
\[ \lambda\cdot u	=f(\lambda) u f(\lambda)^{-1}.\]
\item 
For all $\gamma\in\Gamma$ with $\sz(\gamma)=m$, and all $\lambda\in L_m$
\[f(\gamma\circ \lambda\circ \gamma^{-1})=\gamma\cdot f(\lambda)\]
\end{enumerate}

\begin{proposition}\label{prop:A1}
For $f\colon L\to \mathsf{U}$ as above, the map 
\begin{equation}\label{eq:ltrain} L\to \Gamma\ltimes \mathsf{U},\ \ \lambda\mapsto (\lambda,f(\lambda^{-1}))\end{equation}
is an embedding as a normal subgroupoid of $\Gamma\ltimes \mathsf{U}$. 
\end{proposition}
\begin{proof}
The calculation, using property (a),  
\begin{align*} \big(\lambda',f((\lambda')^{-1}))\big)\circ \big(\lambda,f(\lambda^{-1})\big)&=\big(\lambda'\circ \lambda,\ 
\big(\lambda^{-1}\cdot f((\lambda')^{-1}))\big)\,	f(\lambda^{-1}) \big)\\
  &=\big(\lambda'\circ \lambda,\ 
 f(\lambda^{-1})f((\lambda' )^{-1})\big)\\
 &=\big(\lambda'\circ \lambda,\ 
 f( (\lambda'\circ\lambda )^{-1})\big)
\end{align*}
shows that \eqref{eq:ltrain} is a morphism of groupoids. To show that the image is a normal subgroupoid, 
we compute, using (a), 
\begin{align*}
(\gamma,u)\circ (\lambda,f(\lambda^{-1}))&=
\big(\gamma\circ \lambda,\ (\lambda^{-1}\cdot u)\,f(\lambda^{-1})\big)\\
&=\big(\gamma\circ \lambda,\ f(\lambda^{-1})\,u\big).
\end{align*}
Since $ (\gamma,u)^{-1}=(\gamma^{-1},\gamma\cdot u^{-1})$, it follows that 
\begin{align*}
(\gamma,u)\circ (\lambda,f(\lambda^{-1}))\circ (\gamma,u)^{-1}
&=\Big(\gamma\circ \lambda\circ\gamma^{-1},\ 
\big(\gamma\cdot\big( f(\lambda^{-1})\,u\big)\big)\,\big(\gamma\cdot u^{-1}\big)
\Big)\\
&=\Big(\gamma\circ \lambda\circ\gamma^{-1},\ \gamma\cdot f(\lambda^{-1})\Big)\\
&=\Big(\gamma\circ \lambda\circ\gamma^{-1},\ f(\gamma\circ \lambda^{-1}\circ \gamma^{-1})\Big)
\end{align*}
as required. 
\end{proof}
As a consequence, 
\begin{equation}\label{eq:G} (\Gamma\ltimes \mathsf{U})/L\rra M\end{equation}
is 	a well-defined Lie groupoid. 

\begin{remarks}
	\begin{enumerate}
		\item The projection $\Gamma\ltimes \mathsf{U}\to \Gamma$ descends to a morphism of Lie groupoids
		\[ (\Gamma\ltimes \mathsf{U})/L\to \Gamma/L.\]
		\item Thinking of $\Gamma$ as a principal $L$-bundle $p\colon \Gamma\to \Gamma/L$
		(or rather, family of principal $L_m$-bundles) 
		 over $\Gamma/L$, the groupoid \eqref{eq:G} is an associated group bundle, for the conjugation action of $L_m$ on $\mathsf{U}_m$  via $f\colon L_m\to \mathsf{U}_m$.
		\item The equivalence relation defining the quotient map $\Gamma\ltimes \mathsf{U}\to (\Gamma\ltimes \mathsf{U})/L$ is explicitly given by 
		\[ (\gamma_1,u_1)\sim (\gamma_0,u_0)\ \Leftrightarrow p(\gamma_1)=p(\gamma_0),\ 
		u_1=f(\gamma_1^{-1}\gamma_0)u_0.\]	
	\end{enumerate}
\end{remarks}
\end{appendix}

 \bibliographystyle{amsplain} 

\def\cprime{$'$} \def\polhk#1{\setbox0=\hbox{#1}{\ooalign{\hidewidth
			\lower1.5ex\hbox{`}\hidewidth\crcr\unhbox0}}} \def\cprime{$'$}
\def\cprime{$'$} \def\cprime{$'$} \def\cprime{$'$} \def\cprime{$'$}
\def\polhk#1{\setbox0=\hbox{#1}{\ooalign{\hidewidth
			\lower1.5ex\hbox{`}\hidewidth\crcr\unhbox0}}} \def\cprime{$'$}
\def\cprime{$'$} \def\cprime{$'$} \def\cprime{$'$} \def\cprime{$'$}
\providecommand{\bysame}{\leavevmode\hbox to3em{\hrulefill}\thinspace}
\providecommand{\MR}{\relax\ifhmode\unskip\space\fi MR }
\providecommand{\MRhref}[2]{%
	\href{http://www.ams.org/mathscinet-getitem?mr=#1}{#2}
}
\providecommand{\href}[2]{#2}

\end{document}